\documentclass[11pt]{amsart}
\usepackage[margin=2cm]{geometry}

\usepackage[normalem]{ulem}
\usepackage{amssymb}
\usepackage{amsfonts}
\usepackage{mathrsfs}
\usepackage{cite}
\usepackage{graphicx}
\usepackage{mathtools}
\usepackage{hyperref}

\usepackage{float}

\usepackage{pgfplots}
\pgfplotsset{compat=1.15}
\usepackage{mathrsfs}
\usetikzlibrary{arrows}
\usetikzlibrary{arrows.meta}
\usetikzlibrary{trees}
\usetikzlibrary{intersections,angles,quotes}

\newtheorem{theorem}{Theorem}[section]
\newtheorem{claim}[theorem]{Claim}
\newtheorem{lemma}[theorem]{Lemma}
\newtheorem{proposition}[theorem]{Proposition}

\newtheorem{conjecture}[theorem]{Conjecture}

\DeclareMathOperator{\per}{per}
\DeclareMathOperator{\area}{area}

\title{Penny graphs in the hyperbolic plane}
\author{\'Ad\'am Sagmeister}
\address{Bolyai Institute, University of Szeged, Aradi vértanúk tere 1, H-6720 Szeged,
Hungary} \email{sagmeister.adam@gmail.com}
\thanks{\'Ad\'am Sagmeister is funded by the grant NKFIH 150151. Project no.\ 150151 has been implemented with the support provided by
the Ministry of Culture and Innovation of Hungary from the National
Research, Development and Innovation Fund, financed under the
ADVANCED\_24 funding scheme.}
\author{Konrad J. Swanepoel}
\address{Department of Mathematics, London School of Economics and Political Science, Houghton Street, London, WC2A 2AE, UK}
\email{k.swanepoel@lse.ac.uk}
\thanks{Konrad Swanepoel is partially supported by ERC grant no.\ 882971, ``GeoScape,'' and by the Erd\H os Center.}
\subjclass{Primary 52C15, Secondary 52C10, 51M09}
\keywords{hyperbolic geometry, circle packing, penny graph}

\mathtoolsset{showonlyrefs=true}

\begin{document}

\begin{abstract}
We consider the problem of finding the maximum number $e_d(n)$ of pairs of touching circles in a packing of $n$ congruent circles of diameter $d$ in the hyperbolic plane of curvature $-1$.
In the Euclidean plane, the maximum comes from a spiral construction of the tiling of the plane with equilateral triangles (Harborth 1974), with a similar result in the hyperbolic plane for the values of $d$ corresponding to the order-$k$ triangular tilings (Bowen 2000).
We present various upper and lower bounds for $e_d(n)$ for all values of $d > 0$.
In particular, we prove that if $d > 0.66114\dots$ except for $d=0.76217\dots$, then the number of touching pairs is less than the one coming from a spiral construction in the order-$7$ triangular tiling, which we conjecture to be extremal.
We also give a lower bound $e_d(n) > (2+\varepsilon_d)n$ where $\varepsilon_d > 1$ for all $d > 0$.
\end{abstract}

\maketitle

\section{Introduction}
A \emph{penny graph} is the incidence graph of a finite packing of congruent circles.
In other words, for each circle in the packing, there is a vertex in the penny graph, and we join two vertices if their corresponding circles touch.
Equivalently, such a graph is the \emph{minimum-distance graph} of a finite set of points with minimum distance $d$, and where two points are joined by an edge if their distance is $d$.
In the Euclidean plane, the combinatorics of penny graphs have been well studied, starting with Erd\H{os}'s seminal paper \cite{E46} and Harborth's determination of the maximum number of edges \cite{Harborth1974}, and with other contributions by Kupitz \cite{Kupitz94}, Pach and T\'oth \cite{Pach-Toth96}, T\'oth \cite{Toth97}, Csizmadia \cite{Csizmadia98}, Swanepoel \cite{S02}, Eppstein \cite{Eppstein18}, and Sagdeev \cite{Sagdeev25}.

Harborth showed that the maximum number of edges of a penny graph with $n$ vertices is $\lfloor 3n-\sqrt{12n-3}\rfloor$.
The analogous question can be considered in higher dimensional Euclidean space \cite{Bezdek12}, in normed spaces (see the references in the survey \cite{S-survey}) and for spaces of non-zero constant curvature, namely spherical and hyperbolic spaces.
Note that in the latter cases, the maximum number of edges will depend on the size of the spheres relative to the curvature.
For the $2$-dimensional sphere, Bezdek, Connelly and Kert\'esz \cite{BCK85} improved on the trivial upper bound of $5n/2$, and for the hyperbolic plane, Bowen \cite{B00} showed that a certain construction is optimal if the minimum distance $d$ is such that the plane can be tiled by equilateral triangles of side length $d$.

It is the purpose of this paper to investigate the maximum number of minimum distances for all values of $d > 0$.
Our main results are various upper and lower bounds for the maximum number of edges in a penny graph on $n$ vertices, depending on $d$, as detailed in the next section.
We also give a slightly simpler proof of Bowen's result that when the distance is that of the order-$k$ triangular tiling, the maximum number of edges is attained by a certain spiral construction.
Using the same ideas, we show that a spiral construction is also optimal for maximizing the number of edges in subgraphs of the $\{p,q\}$-tilings in the hyperbolic plane.

Finally, we also consider the simple problem of ``infinite diameter circles'', and find the tight maximum number of touching pairs in a packing of $n$ horocycles.

\section{Our results}
Given \(d > 0\) and \(n\in\mathbb{N}\), let \(e_d(n)\) denote the maximum number of touching pairs in a packing of circles of diameter \(d\) in the hyperbolic plane (of curvature $-1$).
Equivalently, this is the largest number of edges in a minimum-distance graph on \(n\) vertices with minimum distance \(d\) in the hyperbolic plane.
Since $e_d(n)$ is superadditive, it follows from Fekete's lemma that \[\lim_{n\to\infty} e_d(n)/n =: c(d)\] exists.
In fact, it is easy to see that $e_d(m+n)\geq e_d(m)+e_d(n)+2$ if $m\geq 2$ and $n\geq 1$.
This (or a simple construction) immediately gives $e_d(n)\geq 2n-3$ for all $d$.
Since all minimum-distance graphs are planar, we also have $e_d(n)\leq 3n-6$.
Thus, $2\leq c(d)\leq 3$.
In this paper we show that both inequalities are strict and find upper bounds and lower bounds of $c(d)$ for all values of $d$.
The bounds are illustrated in Figure~\ref{fig:summary} and described more precisely below, after introducing some notation.
\begin{figure}
\centering
\includegraphics[scale=1]{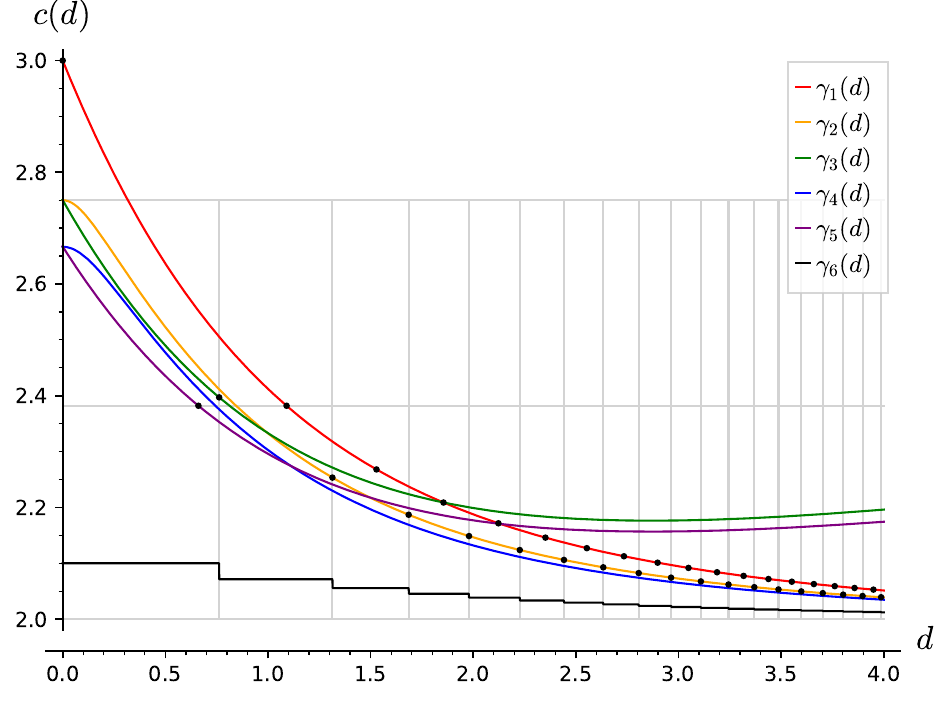}
\caption{Upper and lower bounds for $c(d)$. The red upper bound $\gamma_1(d)$ holds for all $d$ and is tight for the order-$k$ triangular tilings (the dots on the red graph).
The orange $\gamma_2(d)$ and green $\gamma_3(d)$ bounds hold for all $d\neq d(k)$.
The bound $\gamma_3(\overline{d}(6))=2.39698265738619\dots$ is the best we have for $d=\overline{d}(6)$ (see the dot on the green curve), and the bounds $\gamma_2(\overline{d}(k))$ are the best for $d=\overline{d}(k)$, $k\geq 7$ (see the dots on the orange curve).
The blue $\gamma_4(d)$ and purple $\gamma_5(d)$ upper bounds hold for all $d$ such that $d\neq d(k)$, $k\geq 7$ and $d\neq\overline{d}(k)$, $d\geq 6$.
These curves intersect at $d=1.1128036956703866\dots$.
Apart from the semiregular tiling $d=\overline{d}(6)$, we have that Conjecture ~\ref{conj3} holds for all $d > 0.6611380871710578\dots$, the solution of $\gamma_5(d)=3-1/\varphi$ (shown by the dot on the purple curve).
The vertical gray lines are at the values where $d=\overline{d}(k)$.
The horizontal gray lines are at $3-1/\varphi$ and $2.75$.
The black step function is the lower bound $\gamma_6(d)$.}
\label{fig:summary}
\end{figure}
For \(k\geq 7\), let \(d(k)\) denote the side length of an equilateral triangle with angles all equal to \(2\pi/k\).
These triangles are the prototiles of the order-\(k\) triangular tiling of the hyperbolic plane, where each vertex has degree \(k\).

For any distance \(d > 0\) in the hyperbolic plane of curvature $-1$, denote the angle of a regular triangle with side length \(d\) by \(\alpha=\alpha(d)\), and its area by \(A(d)=\pi-3\alpha\).
We can compute \(\alpha\) using the relation \(\cos\alpha = \dfrac{\tanh d/2}{\tanh d}\).
We also let $\alpha_m=\alpha_m(d)$ denote the angle of a regular $m$-gon with side length $d$.
Then $\sin\alpha_m/2 = \cos(\pi/k)/\cosh(d/2)$.
In particular, this gives another way to express the relation between $\alpha=\alpha_3$ and $d$, namely $\sin\alpha/2 = (2\cosh d/2)^{-1}$.
Also, the angle $\alpha_4$ of a square satisfies $\sin\alpha_4/2=\sqrt{2}\sin\alpha/2$.
For any $k\geq 7$, we let $d(k)$ be the value of $d$ such that $\alpha(d)=2\pi/k$, 
and for $k\geq 6$, we let $\overline{d}(k)$ be the value of $d$ such that $\alpha_4(d)+(k-1)\alpha(d)=2\pi$.
Thus, $d(k)$ is the edge length of an order-$k$ triangular tiling, and $\overline{d}(k)$ is the edge length of a semiregular tiling where there are $k-1$ equilateral triangles and a square around each vertex. 
We have $\overline{d}(6) < d(7) < \overline{d}(7) < d(8) < \dots$ (see Table~\ref{table:1}).
\begin{table}
    \centering
\begin{tabular}{ccccc}
$k$ & $\alpha(d(k))=\frac{2\pi}{k}$ & $d(k)$ & $\alpha(\overline{d}(k))$ & $\overline{d}(k)$\\ \hline
$6$ & --- & --- & $0.969004$ & $0.76217$ \\
$7$ & $0.897598$ & $1.09055$ & $0.841836$ & $1.31399$ \\
$8$ & $0.785398$ & $1.52857$ & $0.743463$ & $1.68530$ \\
$9$ & $0.698132$ & $1.85508$ & $0.665375$ & $1.97930$ \\
$10$ & $0.628319$ & $2.12255$ & $0.601989$ & $2.22672$ \\
$11$ & $0.571199$ & $2.35171$ & $0.549554$ & $2.44200$ \\
$12$ & $0.523599$ & $2.55337$ & $0.505480$ & $2.63338$ \\
$13$ & $0.483322$ & $2.73408$ & $0.467925$ & $2.80610$ \\
$14$ & $0.448799$ & $2.89815$ & $0.435550$ & $2.96375$ \\
$15$ & $0.418879$ & $3.04861$ & $0.407355$ & $3.10892$ \\
$16$ & $0.392699$ & $3.18771$ & $0.382582$ & $3.24357$ \\
$17$ & $0.369599$ & $3.31713$ & $0.360645$ & $3.36919$ \\
$18$ & $0.349066$ & $3.43821$ & $0.341084$ & $3.48698$ \\
$19$ & $0.330694$ & $3.55201$ & $0.323533$ & $3.59791$ \\
$20$ & $0.314159$ & $3.65939$ & $0.307699$ & $3.70274$ \\
$21$ & $0.299199$ & $3.76107$ & $0.293341$ & $3.80215$ \\
$22$ & $0.285599$ & $3.85763$ & $0.280263$ & $3.89669$ \\
$23$ & $0.273182$ & $3.94959$ & $0.268300$ & $3.98682$ \\
\end{tabular}
\caption{Numerical values of $d(k)$ and $\overline{d}(k)$ and their associated angles $\alpha$}
\label{table:1}
\end{table}
Our main result is the following.
\begin{theorem}\label{thm:main}
We have the following bounds for $c(d)$:
\begin{align*}
c(d) &\leq \begin{cases}
\gamma_1(d) := \displaystyle \frac{\pi}{\alpha}-\sqrt{\left(\frac{\pi}{\alpha}-1\right)\left(\frac{\pi}{\alpha}-3\right)} & \text{for all $d > 0$}\\[4mm] 
\gamma_2(d) := \displaystyle \frac{\pi}{\alpha}-\frac18-\sqrt{\left(\frac{\pi}{\alpha}\right)^2-\frac{17\pi}{4\alpha}+\frac{241}{64}} & \text{if $d=\overline{d}(k), k\geq 6$}\\[4mm]
\gamma_3(d) := \displaystyle 2+\frac{3(d-\pi+3\alpha)}{4(d+\pi-3\alpha)} & \text{if $d=\overline{d}(k), k\geq 6$}\\[4mm]
\gamma_4(d) := \displaystyle \frac{\pi}{\alpha} - \frac16 - \sqrt{\left(\frac{\pi}{\alpha}\right)^2-\frac{13\pi}{3\alpha} + \frac{145}{36}} & \text{for all $d > 0$ if $d\neq d(k),\overline{d}(k)$, $k\geq 6$}\label{eq:gamma4} \\[4mm]
\gamma_5(d) := \displaystyle \frac{8}{3} - \frac{4(\pi-3\alpha)}{3(d+\pi-3\alpha)} & \text{for all $d > 0$ if $d\neq d(k),\overline{d}(k)$, $k\geq 6$}
\end{cases}\\
\;c(d) &\;\geq \;\;\gamma_6(d) := \displaystyle 2 +\frac{1}{4\lfloor(2\pi-\alpha_4)/\alpha\rfloor - 6} \qquad\quad\;\;\;\text{for all $d > 0$}
\end{align*}
\end{theorem}
The proofs of the various parts of this theorem are contained in Section~\ref{section:upper-bound} and \ref{section:lower-bound}.
Bowen proved \cite{B00} that a certain straightforward spiral construction that cuts out a finite part of any order-$k$ triangular tiling gives a minimum-distance graph on \(n\) vertices with the maximum number \(e_{d(k)}(n)\) of edges.
Rold\'an and Toal\'a-Enr\'{\i}quez \cite{RT22} gave an asymptotic formula for the number of edges on the boundary of certain related spiral graphs.
It follows from these results that $c(d)=\gamma_1(d)$ for all $d=d(k)$, $k\geq 7$, and that the order-\(7\) tiling is optimal among all order-\(k\) triangular tilings.
In particular, \(e_{d(7)}(n) = (4-\varphi)n + O(1)\), where \(\varphi=(1+\sqrt{5})/2\) is the golden ratio.
Our self-contained proof of the first inequality $c(d)\leq\gamma_1(d)$ of Theorem~\ref{thm:main} is given in Proposition~\ref{prop:inductive_bound} below, where we show the following more precise bound.
\begin{theorem}\label{thm:upper-bound}
For all $n\geq 1$,
\[e_{d(7)}(n) \leq \left\lfloor\frac{7}{2}n - \sqrt{\frac54n^2+15n-4}\right\rfloor.\]
\end{theorem}
\begin{conjecture}\label{conj1}
For all $n\geq 1$, $e_{d(7)}(n)$ equals the upper bound from Theorem~\ref{thm:upper-bound}.
\end{conjecture}
As further limited evidence for this conjecture, we have the following observation, proved in the next section.
\begin{proposition}\label{prop:small-values}
Conjecture~\ref{conj1} holds
for all $n\leq 366536$ and for all $n=7F_{2k+1}-6$, $k\geq 0$, where $F_i$ is the $i$th Fibonacci number.
\end{proposition}
We also state the following related conjectures.
\begin{conjecture}\label{conj2}
For a given \(n\), among all \(d > 0\), a spiral with \(n\) vertices of the order-\(7\) triangular tiling has the maximum number of edges. In other words, \(\max_{d > 0} e_d(n) = e_{d(7)}(n)\).
\end{conjecture}
\begin{conjecture}\label{conj3}
For all $d > 0$, if $d\neq d(7)$, then $c(d) < c(d(7))$.
\end{conjecture}
It follows from Theorem~\ref{thm:main} that this conjecture holds for all $d$ that is not too small.
\begin{theorem}
    $c(d) < c(d(7))$ for all $d > 0.6611380871710578\dots$ except possibly $d=0.762173254820934\dots$.
\end{theorem}
It follows from Theorem~\ref{thm:main} that we have a general upper bound $c(d) < 8/3 = \lim_{d\to 0}\gamma_4(d)$, valid for all $d > 0$.

It is somewhat surprising that the trivial lower bound $e_d(n)\geq 2n-3$ can be improved for all $d > 0$.
An equivalent definition of the lower bound $\gamma_6(d)$ is the following:
\[ \gamma_6(d) = 2 + \frac{1}{10}\quad\text{if $0 < d < \overline{d}(6)$,}\]
and for $q\geq 7$,
\[ \gamma_6(d) = 2 + \frac{1}{4q-14}\quad\text{if $\overline{d}(q-1) \leq d < \overline{d}(q)$.}\]
Thus we have that $c(d) > 2$ for all $d > 0$.
Also, since $\lim_{d\to\infty} \gamma_i(d)=2$ for all $i=1,2,4,6$, it follows that $\lim_{d\to\infty}c(d)=2$.

We may ask what happens if we consider circle packings of ``circles of infinite radius'' in the hyperbolic plane.
If we consider a variable point $c$ on a ray $\overrightarrow{ab}$, then the circle with centre $c$ and radius $ac$ converges to a limit curve as the distance $ac\to\infty$ which is called a \emph{horocycle}.
\begin{proposition}\label{prop:horocycles}
In a packing of $n$ horocycles in the hyperbolic plane, the number of touching pairs of horocycles is at most $2n-3$, and this bound is tight.
\end{proposition}
\begin{proof}
In the Poincar\'e model, horocycles are circles in the disc that are tangent with the boundary circle.
Consider a packing of horocycles are circles in the Poincar\'e model with boundary circle $C$.
If we have a packing of circles inside $C$, all touching $C$, then the corresponding contact graph is outerplanar: all of its vertices are on the boundary of its outer face.
It is well known that an outerplanar graph on $n$ vertices has at most $2n-3$ edges.

We observe that it is easy to find such a packing of horocycles with $2n-3$ touching pairs, as indicated in Figure~\ref{fig:horocycles}.
\end{proof}
\begin{figure}[H]
\centering
\begin{tikzpicture}[line cap=round,line join=round,>=triangle 45,x=1cm,y=1cm,scale=4]
\clip(-1.01,-1.01) rectangle (1.01,1.01);
\draw [line width=0.8pt,gray] (0,0) circle (1cm);
\draw [line width=1.2pt] (0,-0.6058339496821427) circle (0.39416605031785734cm);
\draw [line width=1.2pt] (0.5367491801703845,-0.6645214096239817) circle (0.1457820030798791cm);
\draw [line width=1.2pt] (-0.618076929262972,-0.4252880278302017) circle (0.24974071341778958cm);
\draw [line width=1.2pt] (-0.4789648026156955,-0.7536489458224747) circle (0.10709468184056929cm);
\draw [line width=1.2pt] (0.4043615177531534,-0.8363086107557989) circle (0.07126601092558467cm);
\draw [line width=1.2pt] (0.6143785081861296,-0.15938351648343624) circle (0.36529652680408564cm);
\draw [line width=1.2pt] (0.7319310466924195,-0.5765210717399204) circle (0.06832335944912675cm);
\draw [line width=1.2pt] (-0.3645953345426777,-0.870376396820022) circle (0.05646926078555587cm);
\draw [line width=1.2pt] (0.3158071667983274,-0.9054739803525566) circle (0.041173237013726334cm);
\draw [fill=black] (4.765520847944522,-3.4050239751818117) circle (2pt);
\end{tikzpicture}
\caption{}\label{fig:horocycles}
\end{figure}

\section{Preliminaries}
In this section we collect basic facts from hyperbolic geometry and graph theory that will be used in our proofs.
\begin{lemma}\label{lemma:2connected}
In a $2$-connected plane graph with $f_i$ bounded faces with $i$ edges and boundary face with $b$ edges, the number of edges is \(e=3n-3-b-\sum_{i\geq 4} (i-3)f_i\).
\end{lemma}
\begin{proof}
Since \(G\) is connected, Euler's polyhedral formula gives \(n-e+\sum_i f_i = 1\).
Since \(G\) is \(2\)-connected, by counting the edges in each face, \(2e = \sum_i if_i +b\).
It follows that \(e=3n-3-b-\sum_{i\geq 4} (i-3)f_i\). 
\end{proof}

For later reference, we state here the elementary Gauss--Bonnet theorem for polygons in the hyperbolic plane (shown by Lambert).
\begin{lemma}[Gauss--Bonnet]\label{lem:Gauss-Bonnet}
The area of a simple polygon $P$ with $n$ vertices and internal angles $\theta_1,\dots,\theta_n$ in the hyperbolic plane (with curvature $-1$) is equal to $\area(P)=(n-2)\pi-\sum_{i=1}^n\theta_i$.
\end{lemma}
The Euclidean version of the following lemma is usually called Oler's inequality \cite{Oler61, R51}, and it goes back to Thue \cite{Thue}.
It was extended to the hyperbolic plane by B\"or\"oczky Jr.\ \cite{B02}, of which the following lemma is a special case.

\begin{lemma}\label{lemma:oler}
Let $G$ be a $2$-connected plane graph in the hyperbolic plane with each bounded face a triangle, and with each edge of length at least $d$.
Let $B$ denote the boundary polygon of $G$.
Then
\[ \frac{\area(B)}{A(d)} + \frac{\per(B)}{d} \geq 2n-2,\]
where $n$ is the number of vertices of $G$.
Equality holds if and only if each bounded face of $G$ is an equilateral triangle of side length $d$.
\end{lemma}

As a corollary of the above lemma we obtain a strengthening of Bowen’s Lemma 2 in \cite{B00}, leading to a slightly simplified proof of Bowen's result (Theorem~\ref{thm:bowen} in the next section).

\begin{lemma}\label{lemma:area:ngons}
Let \(G\) be a $2$-connected minimum-distance graph with minimum distance $d$ in the hyperbolic plane, with $n$ vertices and $b$ boundary vertices.
Then \(\area(G)\geq(2n-2-b)A(d)\), with equality iff \(G\) can be triangulated into \(n-2\) equilateral triangles of side length \(d\).
\end{lemma}

\begin{lemma}[Hyperbolic isoperimetric inequality, Schmidt \cite{S40}]\label{lemma:isoperimetric}
Suppose that a simple closed curve in the hyperbolic plane has length \(L\) and area \(A\).
Then \(L^2\geq 4\pi A + A^2\), with equality iff the curve is a circle.
\end{lemma}

\section{Proof of Proposition~\ref{prop:small-values}}
This proposition follows from the following description of the finite difference \[\Delta e_{d(7)}(n) := e_{d(7)}(n+1)-e_{d(7)}(n).\]
\begin{proposition}\label{prop:recursion}
Let $A=(a_n)_{n\geq 0}$ be the infinite word defined as the limit of the sequence $(A_i)_{i\geq 0}$ of finite words defined recursively by $A_0=012$ and $A_{i+1}=A_i\cdot2^{[4-a_i]}3$, $i\geq 0$.
Then $\Delta e_{d(7)}(n) = a_n$.
\end{proposition}
In other words, we start with the word $012$, and then working from left to right, for each $0$ we encounter, we add $2,2,2,2,3$ on the right, for each $1$ we add $2,2,2,3$, for each $2$ we add $2,2,3$, and for each $3$ we add $2,3$, as shown in the following diagram.
\begin{center}
\begin{tikzpicture}[
  level 1/.style={sibling distance=4.6cm},
  level 2/.style={sibling distance=4cm},
  level 3/.style={sibling distance=0.9cm},
]

\node {$\bullet$}
  child {node {0}
    child {node {22223}
      child {node {223}}
      child {node {223}}
      child {node {223}}
      child {node {223}}
      child {node {23}}
    }
  }
  child {node {1}
    child {node {2223}
      child {node {223}}
      child {node {223}}
      child {node {223}}
      child {node {23}}
    }
  }
  child {node {2}
    child {node {223}
      child {node {223}}
      child {node {223}}
      child {node {23}}
    }
  };

\end{tikzpicture}
\end{center}

In this way we obtain
\[A = 012\;22223\,2223\,223\;223\,223\,223\,223\,23\;223\,223\,223\,23\;223\,223\,23\dots.\]
By summing $a_n$, we obtain $e_{d(7)}(n) = \sum_{i=0}^{n-1}a_i$, thus giving the following first few values:

\smallskip
\begin{tabular}{c|cccccccccccccccccccc}
    $n$ & 1 & 2 & 3 & 4 & 5 & 6 & 7 & 8 & 9 & 10 & 11 & 12 & 13 & 14 & 15 & 16 & 17 & 18 & 19 & 20 \\ \hline
    $e_{d(7)}(n)$ & 0 & 1 & 3 & 5 & 7 & 9 & 11 & 14 & 16 & 18 & 20 & 23 & 25 & 27 & 30 & 32 & 34 & 37 & 39 & 41
\end{tabular}

\begin{proof}[Proof of Proposition~\ref{prop:small-values}]
To show that $e_{d(7)}(n) = \left\lfloor\frac{7}{2}n - \sqrt{\frac54n^2+15n-4}\right\rfloor$ for all $n\leq 366536$, we computed $e_{d(7)}(n)$ for these values of $n$ using the recursion of Proposition~\ref{prop:recursion}, and compare it to the formula.
(This is very time intensive, which is why the largest value $366536$ is rather modest.)

If we consider complete layers of the hyperbolic order-$7$ triangular tiling, starting with a central vertex, we obtain that the number of vertices is
\[n=7F_{2k+1}-6 = \frac{7}{\sqrt{5}}\varphi^{2k+1}+\frac{7}{\sqrt{5}}\varphi^{-2k-1}-6\]
and the number of edges is $e=7(3F_{2k+1}-F_{2k}-3)$, where $F_i$ is the $i$-th Fibonacci number (see \cite[Sequence A001354]{oeis}).
Writing $e$ in terms of $n$, we obtain
\[e=\left(3-\frac{1}{\varphi}\right)n-3-\frac{6}{\varphi} + 7\varphi^{-2k-1}.\]
We can write $\varphi^{2k+1}$ in terms of $n$ only:
$\varphi^{2k+1}=\frac12\left(\frac{\sqrt{5}}{7}(n+6)+\sqrt{\frac{5}{49}(n+6)^2-4}\right)$
and $\varphi^{-2k-1}=\frac12\left(\frac{\sqrt{5}}{7}(n+6)-\sqrt{\frac{5}{49}(n+6)^2-4}\right)$.
Thus, for complete layers we obtain
\begin{align*}
    e &= \left(4-\varphi\right)n-3-\frac{6}{\varphi} + \frac{98/\sqrt{5}}{n+6+\sqrt{(n+6)^2-14^2}} \\
    &= \frac{7}{2}n - \frac12\sqrt{5(n+6)^2-14^2}.
\end{align*}
Bowen \cite{B00} showed that these complete layers give the maximum number of edges for the same number of vertices (see Theorem~\ref{thm:bowen} below).
\end{proof}

\section{The spiral construction and a generalization of Bowen's upper bound}
We first give a description of the spiral construction for all $\{p,q\}$-tesselations.
Let $p,q\geq 3$ be such that $(p-2)(q-2) > 4$.
Let $G_{p,q}$ be the infinite $\{p,q\}$-tesselation in the hyperbolic plane by regular $p$-gons such that there are exactly $q$ polygons at each vertex.
Let $e=v_1v_2$ be an edge of $G_{p,q}$.
We next define $v_3,v_4,\dots$ inductively, generalizing the spiral graphs of Bowen \cite{B00} where $p=3$.
Assume that $v_1,\dots,v_i$ has already been found.
Let $v_{j_1},\dots,v_{j_t}$ be all the previously defined neighbours of $v_i$, starting with $v_{j_1}=v_{i-1}$, and moving in the clockwise direction around $v_i$, in other words, such that $\triangle v_{j_k}v_{j_{k+1}}v_i$ is positively oriented for each $k=1,\dots,t-1$.
We then let $v_{i+1}$ be the first vertex of $G_{p,q}$ adjacent to $v_i$ after $v_{j_t}$, still moving clockwise around $v_i$.
The spiral graph $G_{p,q}(n)$ is the subgraph of $G_{p,q}$ induced by $v_1,\dots,v_n$.
\begin{figure}[H]
    \centering
    \includegraphics[width=0.5\linewidth]{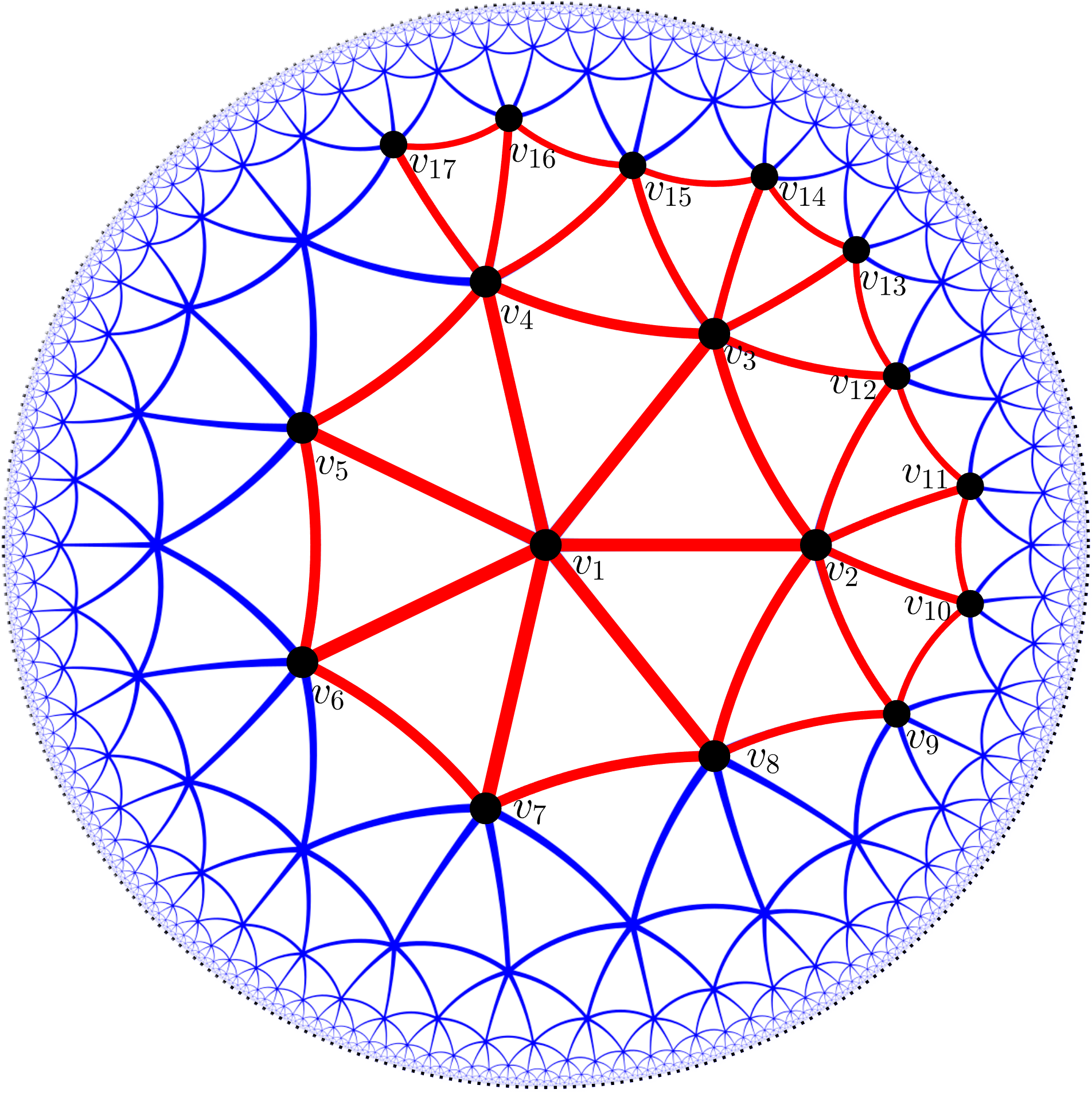}
    \caption{A spiral graph in the order-7 triangular tiling on 17 vertices. (Adapted from a figure from Wikipedia \cite{Taxel})}
    \label{fig:spiral}
\end{figure}
We first give a slightly simpler proof of Bowen's theorem on minimum-distance graphs with distance $d(q)$.
\begin{theorem}[Bowen \cite{B00}]\label{thm:bowen}
    For any $q\geq 7$, the maximum number of edges $e_{d(q)}(n)$ of a minimum-distance graph with distance $d(q)$ is attained by $G_{3,q}(n)$ for all $n\geq 1$.
\end{theorem}
\begin{proof}
    The theorem is trivial if $n\leq 3$, so we assume that $n > 3$ and use induction.
    Consider a minimum-distance graph $G$ with $n$ vertices, distance $d(q)$, and with the maximum number $e=e_{d(q)}(n)$ of edges.
    Two edges cannot intersect, otherwise the triangle inequality would imply that there are two vertices at distance strictly less than $d$.
    Thus, the graph \(G\) is planar.
    If $G$ is not connected, then we can move one of its components towards another component until the distance between the two components becomes $d$, thereby increasing the number of edges.
    Thus $G$ is connected.
    Similarly, $G$ has to be \(2\)-connected, otherwise we can rotate one of its blocks around a cutpoint until a new edge is created.
    By Lemma~\ref{lemma:2connected},
    \begin{equation}\label{eq1}
        e = 3n-3-b-(f_4+2f_5+3f_6+\dots). 
    \end{equation}
    Let $G_s=G_{3,q}(n)$ be the spiral graph on $n$ vertices contained in $G_{3,q}$.
    This is also a minimum-distance graph with distance $d(q)$.
    Then $G_s$ is planar and $2$-connected by construction, and all bounded faces are triangles.
    Lemma~\ref{lemma:2connected} now gives $e_s=3n-3-b_s$, where $e_s$ is the number of edges of $G_s$, and $b_s$ the number of edges on its boundary face.
    By assumption, $e\geq e_s$, from which it follows that $b\leq b_s$, and furthermore, if $b=b_s$, then the spiral construction is optimal, and all optimal graphs have only triangles for bounded faces.
    Thus, to prove the theorem it suffices to show that $b=b_s$.
    We do this by assuming that $b < b_s$ and deriving a contradiction.
    
    We now remove the boundary vertices of $G$ to obtain a smaller graph $G'$ with $n'=n-b$ vertices and $e'$ edges.
    If $b=n$, then \eqref{eq1} gives $e\leq 2n-3$.
    On the other hand, as we add points $v_1,\dots,v_n$ in the spiral graph, from $v_3$ onwards, we add at least $2$ edges for every point, so $e_s\geq 2n-3$.
    Thus $e=e_s$ and $f_4+2f_5+\ldots=0$, hence $b_s=b$, a contradiction.

    Thus without loss of generality, $b < n$.
    Remove the last $b$ vertices $v_{n-b+1},\dots,v_n$ from the boundary of $G_s$ to obtain $G_s'$.
    By induction, $G'$ has at most as many edges $e'$ as the number $e_s'$ of edges of the spiral construction $G_s'$ on $n'$ vertices: $e'\leq e'_s$.

    We next find estimates from above and below of the areas of the boundary polygons of $G$ and $G_s$, which will give us the required contradiction.
    
    We first estimate the angle sum of the boundary polygon of $G$ from below. Each edge in $E(G)\setminus E(G')$ can be assigned to a unique angle of a bounded face of $G$ that is removed when we remove the boundary, with the possible exception of an edge with both endpoints on the boundary of $G$ (a chord of the boundary cycle), which is assigned to two angles (Figure~\ref{fig:angles}).
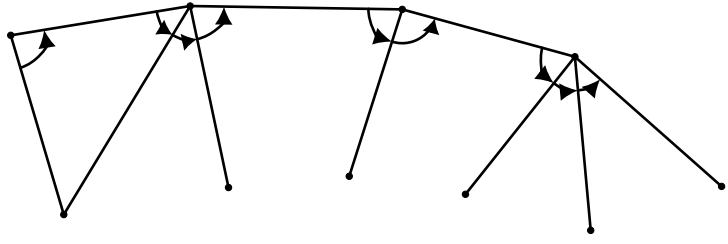
\begin{figure}
\centering
\begin{tikzpicture}[line cap=round,line join=round,scale=0.75]
\draw [line width=1pt] (-8.66,4.82)-- (-5.48,5.34);
\draw [line width=1pt] (-5.48,5.34)-- (-1.72,5.28);
\draw [line width=1pt] (-1.72,5.28)-- (1.34,4.44);
\draw [line width=1pt] (1.34,4.44)-- (3.94,2.14);
\draw [line width=1pt] (1.34,4.44)-- (1.62,1.36);
\draw [line width=1pt] (1.34,4.44)-- (-0.6,2);
\draw [line width=1pt] (-1.72,5.28)-- (-2.66,2.32);
\draw [line width=1pt] (-5.48,5.34)-- (-7.72,1.64);
\draw [line width=1pt] (-7.72,1.64)-- (-8.66,4.82);
\draw [line width=1pt] (-5.48,5.34)-- (-4.8,2.12);
\draw [shift={(-8.66,4.82)},-{Latex[scale=0.8, length=3mm, width=3mm]},line width=1pt] (-73.53245320699729:0.6) arc (-73.53245320699729:9.286927512496993:0.6);
\draw [shift={(-5.48,5.34)},-{Latex[scale=0.8, length=2.5mm, width=3mm]},line width=1pt] (-170.713072487503:0.6) arc (-170.713072487503:-121.1909392412187:0.6);
\draw [shift={(-5.48,5.34)},-{Latex[scale=0.8, length=2.5mm, width=3mm]},line width=1pt] (-121.19093924121873:0.6) arc (-121.19093924121873:-78.07547606238809:0.6);
\draw [shift={(-1.72,5.28)},-{Latex[scale=0.8, length=3mm, width=3mm]},line width=1pt] (179.08578323955103:0.6) arc (179.08578323955103:252.38183994942503:0.6);
\draw [shift={(-5.48,5.34)},-{Latex[scale=0.8, length=3mm, width=3mm]},line width=1pt] (-78.07547606238808:0.6) arc (-78.07547606238808:-0.9142167604489821:0.6);
\draw [shift={(-1.72,5.28)},-{Latex[scale=0.8, length=2.5mm, width=3mm]},line width=1pt] (-107.61816005057503:0.6) arc (-107.61816005057503:-15.35013649242442:0.6);
\draw [shift={(1.34,4.44)},-{Latex[scale=0.8, length=3mm, width=3mm]},line width=1pt] (164.6498635075756:0.6) arc (164.6498635075756:231.51242302928787:0.6);
\draw [shift={(1.34,4.44)},-{Latex[scale=0.8, length=3mm, width=3mm]},line width=1pt] (-128.48757697071213:0.6) arc (-128.48757697071213:-84.80557109226518:0.6);
\draw [shift={(1.34,4.44)},-{Latex[scale=0.8, length=3mm, width=3mm]},line width=1pt] (-84.8055710922652:0.6) arc (-84.8055710922652:-41.49646835521555:0.6);
\draw [fill=black] (-8.66,4.82) circle (1.6pt);
\draw [fill=black] (-5.48,5.34) circle (1.6pt);
\draw [fill=black] (-1.72,5.28) circle (1.6pt);
\draw [fill=black] (1.34,4.44)  circle (1.6pt);
\draw [fill=black] (3.94,2.14)  circle (1.6pt);
\draw [fill=black] (1.62,1.36)  circle (1.6pt);
\draw [fill=black] (-0.6,2)     circle (1.6pt);
\draw [fill=black] (-2.66,2.32) circle (1.6pt);
\draw [fill=black] (-7.72,1.64) circle (1.6pt);
\draw [fill=black] (-4.8,2.12)  circle (1.6pt);
\end{tikzpicture}
\caption{The arrow indicates which edge is associated to an angle}\label{fig:angles}
\end{figure}

    Thus the total number of these angles is at least $e-e'$, which gives that the angle sum of the boundary polygon $B$ of $G$ is at least $(e-e')(2\pi/q)$, and it follows from the Gauss--Bonnet formula (Lemma~\ref{lem:Gauss-Bonnet}) that the area of this polygon satisfies
    \[ \area(B)\leq (b-2)\pi - (e-e')(2\pi/q).\]

   Next, we bound the angle sum of the boundary polygon of the spiral construction $G_s$ from above.
   The total number of angles at the $b$ vertices $v_{n-b+1},\dots,v_n$ that are removed from $G_s$ is at most  $e_s-e_s'-1$.
   This can be seen as follows.
    Since $b < b_s$, we do not remove the whole boundary of $G_s$, and as above, by pairing up boundary angles to edges that are removed, it follows that there are at least $e_s-e_s'-1$ angles.

    Each boundary vertex of $G_s$ that is not removed, has degree at most $4$ in $G_s$ (see \cite[Proposition 3.1.]{MRT23}),
    with the exception of the boundary vertex $v_m$ ($m \leq n-b$) adjacent to $v_n$, which has degree at most $q-1$.
    Thus the angle sum of the boundary polygon $B_s$ of $G_s$ is at most
    \[(e_s-e_s'-1)(2\pi/q) + (b_s-b-1)(6\pi/q) + (q-2)(2\pi/q) = (e_s-e_s'+3b_s-3b+q-6)(2\pi/q).\]
    By Gauss--Bonnet, the area of $G_s$ is bounded below as follows:
    \[\area(B_s) \geq (b_s-2)\pi - (e_s-e_s'+3b_s-3b+q-6)(2\pi/q).\]
    Then
    \begin{equation}\label{eq:areadiff1}
    \area(B)-\area(B_s) \leq (b-b_s)(\pi-6\pi/q) + (e_s-e+e'-e_s'+q-6)(2\pi/q) \leq (1-8/q)\pi,
    \end{equation}
    where the last inequality follows from $b < b_s$, $e_s < e$, and $e'\leq e_s'$.

    We now use Lemma~\ref{lemma:area:ngons} to estimate the two areas.
    First of all, since $G_s$ consists only of triangles, we get the equality
    \[\area(B_s)=(2n-2-b_s)(\pi-6\pi/q),\]
    while for $G$ we obtain the lower bound
    \begin{equation*}
    \area(B) \geq  (2n-2-b)(\pi-6\pi/q).
    \end{equation*}
    Thus,
    \[\area(B)-\area(B_s) \geq (b_s-b)(\pi-6\pi/q)\geq(1-6/q)\pi,\]
    which contradicts \eqref{eq:areadiff1}.
\end{proof}

We next show that an argument similar to the previous proof shows that the spiral construction gives the maximum number of edges among all subgraphs of the $\{p,q\}$-tiling of a fixed number of vertices, for any $p,q\geq 3$ satisfying $(p-2)(q-2) > 4$.
A weaker, asymptotic version of this was shown by Higuchi and Shirai \cite{HS03} and H\"aggstr\"om, Jonasson and Lyons \cite{HJL02}.
A statement equivalent to this was shown by Rold\'an and Toal\'a-Enr\'{\i}quez \cite{RT22}.
Our proof is again simpler.

\begin{lemma}
Consider a cycle $C=v_1\dots v_k$ in the $\{p,q\}$-tesselation, and let $e_i$ be the number of edges in the interior of the cycle emanating from $v_i$.
Then the number of faces inside $C$ is \[\frac{(k-2)(q-2)-4-2\sum_{i=1}^k e_i}{(p-2)(q-2)-4}.\]
\end{lemma}

\begin{proof}
In a $\{p,q\}$-tessellation, each face is a regular $p$-gon with angles $\frac{2\pi}{q}$. Such a $p$-gon has area $(p-2)\pi-p\frac{2\pi}{q}$. The interior angle of $C$ at the vertex $v_i$ is $(e_i+1)\frac{2\pi}{q}$, as the number of regular $p$-gons meeting at $v_i$ is $e_i+1$. Hence, the area of the $k$-gon $C$ is $(k-2)\pi-\frac{2\pi}{q}\sum_{i=1}^k(e_i+1)$. The number of faces is the ratio of the area of $C$ and the area of a regular $p$-gon. Dividing the two areas concludes the proof.
\end{proof}

\begin{theorem}
Let $G$ be a subgraph of the $\{p,q\}$-tesselation with $n$ vertices, where $p > 3$.
Then the number of edges of $G$ is at most the number of edges of the spiral graph $G_{p,q}(n)$.
\end{theorem}

\begin{proof}
Let $G_s=G_{p,q}(n)$ be the subgraph of the $\{p,q\}$-tesselation with $n$ vertices obtained from the spiral construction. 
We use induction on $n$ to show that $e(G)\leq e(G_s)$, with the statement clearly true for $n \leq p$.
Assume that $G$ satisfies $e(G) > e(G_s)$. 
Without loss of generality, $G$ is an induced subgraph of the $\{p,q\}$-tiling.

If $G$ is not connected, we can increase the number of edges by moving a component closer to another component until a new edge is created.
If $G$ is not $2$-connected, then $G=G_1\cup G_2$ with $G_1\cap G_2$ consisting of a single point.
We replace $G_1$ and $G_2$ with spiral constructions $S_1$ and $S_2$ with the same number of vertices as $G_1$ and $G_2$, respectively.
On each $S_i$ there exist $3$ adjacent points $a_i, b_i, c_i$ on the boundary such that $S_i$ is contained in the angle $\angle a_i b_i c_i$.
We can then move $S_2$ such that the edges $a_1b_1$ and $a_2b_2$ are identified, to give a construction on $n-1$ vertices with $e(G)-1$ edges.
Thus, $e(G)-1\leq e(G_s^-)$, the number of edges of the spiral construction $G_s^-$ with $n-1$ vertices.
Since $e(G_s)\geq e(G_s^-)+1$, we have shown that $e(G)\leq e(G_s)$, contradicting our assumption.

Thus, $G$ is $2$-connected.
We next compare $G$ with the spiral construction $G_s$ on $n$ vertices.
By the Euler formula and double counting, we obtain $e(G)\leq \frac{p}{p-2}(n-1)-\frac{b}{p-2}$, where $b$ is the length of the boundary polygon $B$ of $G$.
Similarly, because all bounded faces of $G_s$ are $p$-gons, we have $e(G_s) = \frac{p}{p-2}(n-1) -\frac{b_s}{p-2}$, where $b_s$ is the length of the boundary of $G_s$.
Since $e(G) > e(G_s)$ by assumption, we get $b < b_s$.

Since $\alpha_{p}=2\pi/q$ is the angle of a face of the $\left\{p,q\right\}$-tiling,
the area of the boundary polygon $B_s$ of $G_s$ is
\begin{equation}\label{eq:Gs-area}
\area\left(B_s\right)=\left(\left(p-2\right)\pi-p\cdot\alpha_{p,q}\right)\left(e(G_s)-n+1\right)=\left(\left(p-2\right)\pi-p\cdot\alpha_{p,q}\right)\left(\tfrac{2}{p-2}(n-1) -\tfrac{b_s}{p-2}\right)
\end{equation}
from the Euler formula and the Gauss--Bonnet formula (Lemma~\ref{lem:Gauss-Bonnet}).

Now consider $G$. We remove its boundary to obtain the graph $G'$. with $n-b$ vertices and $e'$ edges. Let $G'_s$ be the graph corresponding to the spiral construction on $n-b$ vertices. Then, by the induction hypothesis, we have $e(G'_s)\geq e(G')$. We also give the following lower bound for the area of the boundary polygon $B_s$ of $G_s$ using similar observations as in the proof of Theorem~\ref{thm:bowen}. In this case, since $p\geq 4$, each boundary vertex that is not removed has degree at most $3$, except from the boundary vertex $v_m$ ($m\leq n-b$) adjacent to $v_n$ that has degree at most $q$. Hence, the angle sum of $B_s$ is at most
\[
\left(e(G_s)-e(G'_s)-1\right)\cdot\alpha_{p,q}+\left(b_s-b-1\right)\cdot 2\alpha_{p,q}+\left(q-1\right)\alpha_{p,q},
\]
thus
\begin{equation}\label{eq:Gs-area-upperbound}
\area\left(B_s\right)\geq\left(b_s-2\right)\pi-\left(\left(e(G_s)-e(G'_s)-1\right)+2\left(b_s-b-1\right)+\left(q-1\right)\right)\cdot\alpha_{p,q}.
\end{equation}
We also observe that $\left(e\left(G\right)-e\left(G'\right)\right)\alpha_{p,q}$ is a lower bound for the angle sum of the boundary polygon $B$ of $G$, as it has at least $e\left(G\right)-e\left(G'\right)$ vertices and each of its angles is a multiple of $\alpha_{p,q}$.
Then,
\begin{equation}\label{eq:G-area-lowerbound}
\area\left(B\right)\leq\left(b-2\right)\pi-\left(e\left(G\right)-e\left(G'\right)\right)\alpha_{p,q}.
\end{equation}
We also want to give a lower bound for $\area\left(B\right)$. To do that, we observe that a $k$-face has at least $\frac{k-2}{p-2}$ copies of the regular $p$-gon, hence
\begin{equation}\label{eq:G-area-upperbound}
\area\left(B\right)\geq\left(\pi-\frac{p}{p-2}\alpha_{p,q}\right)\sum_{k\geq p}\left(k-2\right)f_k,
\end{equation}
where $f_k$ denotes the number of $k$-faces in $G$. We have
\[
\sum_{k\geq p}f_k=1+e\left(G\right)-n
\]
from Euler's formula, while
\[
\sum_{k\geq p}k\cdot f_k=2e\left(G\right)-b,
\]
and hence
\begin{equation*}\label{eq:sumfk}
\sum_{k\geq p}(k-2)\cdot f_k=2n-b-2,
\end{equation*}
which we use to rewrite \eqref{eq:G-area-upperbound} as
\begin{equation}\label{eq:G-area-upperbound2}
\area\left(B\right)\geq\left(\pi-\frac{p}{p-2}\alpha_{p,q}\right)(2n-b-2).
\end{equation}
From \eqref{eq:Gs-area-upperbound} and \eqref{eq:G-area-lowerbound}, we deduce
\begin{equation}\label{eq:GGs-difference-lowerbound}
\area\left(B_s\right)-\area\left(B\right)\geq\left(b_s-b\right)\left(\pi-2\alpha_{p,q}\right)+\left(e\left(G\right)-e\left(G_s\right)+e\left(G'_s\right)-e\left(G'\right)+4-q\right)\alpha_{p,q}.
\end{equation}
On the other hand, from \eqref{eq:Gs-area} and \eqref{eq:G-area-upperbound2}, we have
\begin{equation}\label{eq:GGs-difference-upperbound}
\area\left(B_s\right)-\area\left(B\right)\leq\left(\pi-\frac{p}{p-2}\alpha_{p,q}\right)(b-b_s)\leq\frac{p}{p-2}\alpha_{p,q}-\pi,
\end{equation}
where the last inequality comes from the assumption $b<b_s$. Comparing the inequalities \eqref{eq:GGs-difference-lowerbound} and \eqref{eq:GGs-difference-upperbound}, we get $p\leq 3$, a contradiction.
\end{proof}

\section{Upper bounds}\label{section:upper-bound}
The following induction proof (similar to Harborth's proof \cite{Harborth1974} for penny graphs in the Euclidean plane) works for all $d > 0$, and is tight for all $d(k)$, $k\geq 7$.

\begin{proposition}\label{prop:inductive_bound}
For all $d > 0$ and $n\geq 1$,
\[e_{d}(n) \leq \frac{\pi}{\alpha}n - \sqrt{\left(\frac{\pi}{\alpha}-1\right)\left(\frac{\pi}{\alpha}-3\right)n^2+6\left(\frac{\pi}{\alpha}-1\right)n-\frac{2\pi}{\alpha}+3} < \gamma_1(d)n,\]
where $\gamma_1(d)=\frac{\pi}{\alpha}-\sqrt{(\frac{\pi}{\alpha}-1)(\frac{\pi}{\alpha}-3)}$ is the smaller root of the quadratic polynomial $p_\alpha(x)=\alpha x^2-2\pi x+4\pi-3\alpha$.
\end{proposition}
\begin{proof}
Let \[s(x)=\sqrt{\left(\frac{\pi}{\alpha}-1\right)\left(\frac{\pi}{\alpha}-3\right)x^2+6\left(\frac{\pi}{\alpha}-1\right)x-\frac{2\pi}{\alpha}+3}.\]
Let $G$ be an $n$-vertex penny graph with distance $d$ with maximum number $e$ of edges. It can be easily checked
that if $n=1,2,3,4$, then $e=0,1,3,5$, and these numbers are bounded above by
$\frac{\pi}{\alpha}n - s(n)$.
By maximality, $G$ is 2-connected, since otherwise we can manipulate the graph to increase the number of edges, as in the proof of Theorem~\ref{thm:bowen}. From now on, let $n\geq 5$. We assume by induction that the number of edges of a graph on $n'<n$ vertices is bounded above by
$\frac{\pi}{\alpha}n' - s(n')$.
Let the number of edges of the boundary polygon be $b$.
By Lemma \ref{lemma:2connected} we have $e\leq 3n-3-b$ and this will be at most
$\frac{\pi}{\alpha}n - s(n)$
if
$b\geq \left(3-\frac{\pi}{\alpha}\right)n-3+s(n)$.
Thus we may assume without loss of generality that
\[b<\left(3-\frac{\pi}{\alpha}\right)n-3+s(n).\]
Note that
$\left(3-\frac{\pi}{\alpha}\right)n-3+s(n) < n$
(as it is equivalent to the inequality $0<(n-3)^2+\left(\frac{2\pi}{\alpha}-3\right)$ which clearly holds as $\alpha<\frac{\pi}{3}$). Thus if we remove the boundary vertices, there will be vertices left on which we can apply induction. In total, we remove $r$ edges ($b$ edges from the boundary and $r-b$ edges connecting boundary vertices with interior vertices). Then the angle sum of the boundary polygon $B$ is
\[
(b-2)\pi-\area(B)\geq\alpha r.
\]
By Lemma~\ref{lemma:area:ngons},
\[
\area(B)\geq(2n-2-b)(\pi-3\alpha),
\]
so we obtain $r\leq(2b-2n)\frac{\pi}{\alpha}+3(2n-2-b)$. Thus the number of edges
\[
e\leq r+(n-b)\frac{\pi}{\alpha}-s(n-b),
\]
and this will be at most
$\frac{\pi}{\alpha}n - s(n)$
if
\[
\left(\frac{\pi}{\alpha}-3\right)b+\left(6-\frac{2\pi}{\alpha}\right)n-6-s(n-b)\leq -s(n).
\]
Since the left-hand side is increasing in $b$ (as $b<n$ and $0<\alpha<\frac{\pi}{3}$), it is sufficient to observe that if
\[
x=\left(3-\frac{\pi}{\alpha}\right)n-3+s(n),
\]
then
\[
\left(\frac{\pi}{\alpha}-3\right)x+\left(6-\frac{2\pi}{\alpha}\right)n-6-s(n-x)= -s(n),
\]
or equivalently,
\[
\left(\left(\frac{\pi}{\alpha}-3\right)x+\left(6-\frac{2\pi}{\alpha}\right)n-6+s(n)\right)^2= s(n-x)^2,
\]
that can be checked by expanding both sides.
\end{proof}

We conjecture that, just as with the order-$7$ triangular tiling, the maximum number of edges for all distances $d(k)$ is found by rounding the above upper bound.
\begin{conjecture}\label{conjk}
For all $n\geq 1$ and all $k\geq 7$,
\[e_{d(k)}(n) = \left\lfloor\frac{\pi}{\alpha}n - \sqrt{\left(\frac{\pi}{\alpha}-1\right)\left(\frac{\pi}{\alpha}-3\right)n^2+6\left(\frac{\pi}{\alpha}-1\right)n-\frac{2\pi}{\alpha}+3}\right\rfloor.\]
\end{conjecture}
We now consider the case when $d\neq d(k)$ for all $k\geq 7$.
Consider a minimum-distance graph \(G\) on \(n\) points with edge length \(d\) and \(e=e_d(n)\) edges.
As before, we assume that $G$ is planar and $2$-connected.
From now on, we also assume that \(d\neq d(k)\) for all \(k\geq 7\).
Then \(\alpha(d)\neq 2\pi/k\) for all \(k\geq 7\), and it follows that each interior vertex belongs to a non-triangular face. This, combined with Lemma~\ref{lemma:2connected}, already gives the following estimate.

\begin{proposition}\label{prop:upper-bound1}
If \(d\neq d(k)\) for all \(k\geq 7\), then \(e_d(n) \leq \frac{11}{4}n-\frac34 b-3\). 
\end{proposition}

\begin{proof}
By Lemma~\ref{lemma:2connected}, \(e=3n-3-b-\sum_{i\geq 4} (i-3)f_i\).
Since each interior vertex belongs to a non-triangular face, \(b+\sum_{i\geq 4} if_i\geq n\), and it follows that \(e\leq 3n-3-b-(n-b)/4 \leq \frac{11}{4}n-\frac{3}{4}b -3 \). 
\end{proof}
We can improve this estimate by bounding the area of the boundary polygon from below using Lemma~\ref{lemma:area:ngons} and from above by the area formula if $d$ is large or by the hyperbolic isoperimetric inequality if $d$ is small.

\begin{proposition}\label{prop:upper-bound2}
If \(d\neq d(k)\) for all \(k\geq 7\),
then for all $n\geq 1$,
\[e_d(n) < \min\{\gamma_2(d)n, \gamma_3(d)n\},\]
where
\[\gamma_2(d)=\frac{\pi}{\alpha}-\frac18-\sqrt{\left(\frac{\pi}{\alpha}\right)^2-\frac{17\pi}{4\alpha}+\frac{241}{64}}\]
and
\[\gamma_3(d)=2+\frac{3(d-\pi+3\alpha)}{4(d+\pi-3\alpha)}.\]
\end{proposition}

\begin{proof}
We show the first bound by using induction and Proposition~\ref{prop:upper-bound1}.
Note that $\gamma_2(d)$ is strictly decreasing and in the interval $(2,3)$ as $\alpha$ ranges between $0$ and $\pi/3$.
Let $G$ be an $n$-vertex penny graph with distance $d$ with maximum number $e$ of edges.
Then $G$ is $2$-connected, since otherwise we can manipulate the graph to increase the number of edges.
It can easily be checked that if $n=1, 2, 3, 4$, then $e=0,1,3,5$, respectively, and then $e\leq 2n-2 < \gamma_2(d)n$.
We assume by induction that the number of edges of a graph on $n' < n$ vertices is bounded above by $\gamma_2(d)n'$.
Let the number of edges of the boundary polygon be $b$.
By Proposition~\ref{prop:upper-bound1} we have $e\leq \frac{11}{4}n-\frac34b-3$, and this will be at most $\gamma_2(d)n$ if $b\geq (\frac{11}{3}-\frac43\gamma_1(d))n-4$.
Thus we may assume without loss of generality that $b < (\frac{11}{3}-\frac43\gamma_1(d))n-4$.

Let $r$ be the number of edges of $G$ with at least one endpoint on the boundary polygon $B$.
Observe that we can put these edges in one-to-one correspondence with all of the angles made by two edges with a common endpoint on $B$ and with no edge inside the angle.
Since each one of these angles has measure at least $\alpha$, we obtain that the angle sum of the boundary polygon $B$ is \[(b-2)\pi-\area(B)\geq\alpha r.\]
By Lemma~\ref{lemma:area:ngons}, \[\area(B)\geq(2n-2-b)(\pi-3\alpha),\]
so we obtain an upper bound \[r\leq -(\tfrac{2\pi}{\alpha}-6)n+(\tfrac{2\pi}{\alpha}-3)b-6.\]
If we remove the boundary vertices and their incident edges from $G$, there are $e-r \leq\gamma_2(d)(n-b)$ edges remaining.
Thus the number of edges is bounded above as follows: \[e\leq \gamma_2(d)(n-b)+r\leq \gamma_2(d)(n-b)-(\tfrac{2\pi}{\alpha}-6)n+(\tfrac{2\pi}{\alpha}-3)b-6,\] and this upper bound will be at most $\gamma_2(d)n$ if \[(\tfrac{2\pi}{\alpha}-3-\gamma_2(d))b\leq(\tfrac{2\pi}{\alpha}-6)n+6.\]
Since $b < (\frac{11}{3}-\frac43\gamma_1(d))n-4$, it is sufficient to show that \[(\tfrac{2\pi}{\alpha}-3-\gamma_2(d))(\tfrac{11}{3}-\tfrac43\gamma_2(d))n-4)\leq (\tfrac{2\pi}{\alpha}-6)n+6.\]
This follows from $(\tfrac{2\pi}{\alpha}-3-\gamma_2(d))(\frac{11}{3}-\frac43\gamma_2(d)) = \frac{2\pi}{\alpha}-6$.

For the second bound, we use the isoperimetric inequality which implies that $A < L = bd$.
By Lemma~\ref{lemma:area:ngons}, the area $A$ of the boundary polygon is
\begin{equation}\label{Alb}
A\geq (\pi-3\alpha)(2n-2-b). 
\end{equation}
We obtain $b > \frac{2\pi-6\alpha}{d+\pi-3\alpha}(n-1)$, and we find the upper bound $e\leq \gamma_3(d)n$ as before.
\end{proof}

We next combine both the isoperimetric inequality and the inductive proof with a charging argument to obtain two better upper bounds if we also assume that $d$ is not one of the distances $\overline{d}(k)$ where $k-1$ equilateral triangles and a square fit around a point.

\begin{proposition}\label{prop:upper-bound3}
If \(d\neq d(k)\) for all \(k\geq 7\) and $d\neq \overline{d}(k)$ for all $k\geq 6$, then
\[e_d(n)\leq \min\{\gamma_4(d)n,\gamma_5(d)n\},\] where
\begin{equation}\label{eq:gamma4}
\gamma_4(d) = \frac{\pi}{\alpha} - \frac16 - \sqrt{\left(\frac{\pi}{\alpha}\right)^2-\frac{13\pi}{3\alpha} + \frac{145}{36}}
\end{equation}
and
\[\gamma_5(d) = \frac{8}{3} - \frac{4(\pi-3\alpha)}{3(d+\pi-3\alpha)}.\]
\end{proposition}
Note that \[\frac{\pi}{\alpha} - \frac16 - \sqrt{\left(\frac{\pi}{\alpha}\right)^2-\frac{13\pi}{3\alpha} + \frac{145}{36}} = 2 + \frac{\alpha}{3\pi} + O(\alpha^2)\quad\text{as $d\to\infty$,}\]
which can be compared to the lower bound in Theorem~\ref{thm:lower-bound}.
\begin{proof}
As before, we can assume that $G$ is $2$-connected, and since it is easily checked for $n=1,2$, we also assume that $n\geq 3$, and that the statement holds for all $2\leq n' < n$.
In particular, there is a boundary polygon with $b$ edges, for each $i\geq 3$ we let the number of $i$-faces be $f_i$, and by the Euler formula we have $e=2n-3-b-\sum_{i\geq 4}(i-3)f_i$.

We use a charging argument, and start off by giving each of the $n-b$ interior vertices a charge of $1$.
Since $\alpha\neq 2\pi/q$ for any $q\geq 7$, it follows that each vertex is incident with a non-triangular face.
We then move the charge of each vertex to its incident non-triangular faces by the following two rules:
\begin{enumerate}
\item If all non-triangular faces incident to an interior vertex are quadrilaterals, then we distribute the charge equally among them.
\item Otherwise we distribute the charge equally among all incident faces that are not triangles or quadrilaterals.
\end{enumerate}
It is clear that any bounded $i$-gon, $i\geq 4$, receives a charge of at most $i$.
\begin{claim}
All quadrilaterals receive a charge of at most $3$. 
\end{claim}
\begin{proof}[Proof of the Claim]
Since a quadrilateral $Q=ABCD$ in a penny graph is a rhombus of side length $d$, its opposite angles are equal.
Let $\angle ABC = \gamma \leq \delta = \angle BCD$ be the two angles.

If only two vertices of $Q$ give their full charge to $Q$, then the total charge has to be at most $1+1+\frac12+\frac12=3$.
Thus without loss of generality, at least three vertices of $Q$ give a charge of $1$ to $Q$.
It follows that both $2\pi-\gamma$ and $2\pi-\delta$ are multiples of $\alpha$, say $k\alpha+\gamma=2\pi=\ell\alpha+\delta$.
Thus $|k-\ell|\alpha=|\gamma-\delta|=\delta-\gamma$.
We next show that $\delta-\gamma < \alpha$.
Note that $\gamma > \alpha$ since the diagonals of $Q$ must be strictly larger than $d$.
Let $C'$ be the point on the diagonal $AC$ such that $AC'=d$.
Since $DA=DC=d$, it follows that $DC' < d$.
Therefore, $\delta/2=\angle DAC < \alpha$, so $\delta < 2\alpha$ and $\delta-\gamma < \alpha$.

It follows that $k=\ell$, so $Q$ is a square with outer angles a multiple of $\alpha$, which contradicts our assumption that $d\neq\overline{d}(k)$.
\end{proof}
Since the total charge is $n-b$, we obtain
\[n-b\leq 3f_4 + 5f_5 + 6f_6 +\dots\leq 3\sum_{i\geq 4}(i-3)f_i.\]
By Lemma~\ref{lemma:2connected}, we obtain $e\leq \frac{8}{3}n-\frac{2}{3}b-3$.
Then $e\leq \gamma_4(d)n$ will follow as long as $b\geq(4-\frac32 \gamma_4(d))n-\frac92$.
On the other hand, we obtain just as in the proof of Proposition~\ref{prop:upper-bound2} by induction that $e\leq \gamma_4(d)n$ as long as $(\tfrac{2\pi}{\alpha}-3-\gamma_4(d))b\leq(\tfrac{2\pi}{\alpha}-6)n+6$.
Thus to cover both cases we need $\gamma_4(d)$ to satisfy \[\left(4-\frac32 \gamma_4(d)\right)n-\frac92\leq\frac{(2\pi/\alpha-6)n+6}{2\pi/\alpha-3-\gamma_4(d)}\]
for all $n$.
It is sufficient to take $\gamma_4(d)$ to be the smaller root of \[\left(4-\frac32x\right)(2\pi/\alpha-3-x)-2\pi/\alpha+6=0,\] which is given by \eqref{eq:gamma4}.

To show $e_d(n)\leq \gamma_5(d)n$, we use the isoperimetric inequality $bd > A$, where $A$ is the area of the polygon enclosed by the boundary cycle.
By Lemma~\ref{lemma:area:ngons}, $A\geq(2n-2-b)(\pi-3\alpha)$.
It follows that $b > \frac{2\pi-6\alpha}{d+\pi-3\alpha}(n-1)$.
Together with the charging bound $e \leq\frac83 n-\frac23 b - 3$, we obtain $e < \gamma_5(d)n$.
This finishes the proof of Proposition~\ref{prop:upper-bound3}.
\end{proof}


\section{Lower bounds}\label{section:lower-bound}

Let $d > 0$ be an arbitrary distance in the hyperbolic plane.
It is easy to pack $n$ circles of diameter $d$ such that there are $2n-3$ touching pairs.
Thus, $e_d(n)\geq 2n-3$ for all $d > 0$ and all $n\geq 3$.

We now describe a better construction.
Recall that $\alpha_4 = 2\arcsin(\sqrt{2}\sin(\alpha/2))$ is the angle of a square of side length $d$.

\begin{theorem}\label{thm:lower-bound}
For any $n\in\mathbb{N}$ and $d > 0$, we have $e_d(n) \geq \left(2+\frac{1}{4q-14}\right)n-4$, where $q=2+\lfloor(2\pi-\alpha_4)/\alpha\rfloor$.
Thus \[ c(d) \geq \gamma_6(d) = 2+\frac{1}{4q-14} = 2 +\frac{\alpha}{8\pi} +O(\alpha^2)\quad\text{as $d\to\infty$.}\]
\end{theorem}
\begin{proof}
We first fix and arbitrary $q\geq 6$.
The specific value of $q$ given in the statement of the theorem will follow from the constraints that the construction given below forms a penny graph.
Let $\triangle A_0BC$ be equilateral triangle with side $d$.
Let points $A_1,\dots,A_{q-2}$ be at distance $d$ from $B$ in clockwise order all at distance $d$ around $B$, with $\triangle A_iA_{i+1}B$ equilateral, $i=0,\dots,q-3$.
Similarly, let $A_{-1},\dots,A_{-q+2}$ be at distance $d$ from $C$ in anticlockwise order at distance $d$ around $C$, with  $\triangle A_{-i}A_{-i-1}B$ equilateral, $i=0,\dots,q-3$.
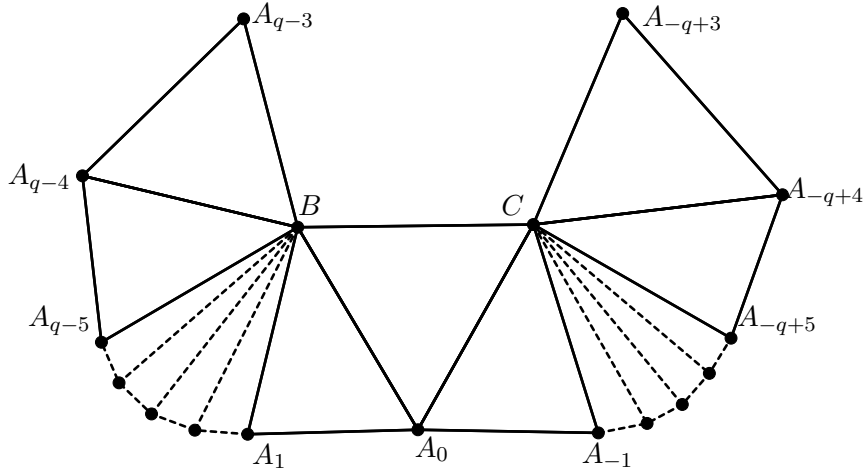
\begin{figure}
\begin{tikzpicture}[line cap=round,line join=round,>=triangle 45,scale=0.9]
\draw[line width=1pt] (2,5.12) -- (-1.48,5.08) -- (0.2946410161513772,2.0862315948301524) -- cycle;
\draw[line width=1pt] (-1.48,5.08) -- (-2.22,2.02) -- (0.2946410161513772,2.0862315948301524) -- cycle;
\draw[line width=1pt] (2,5.12) -- (2.96,2.04) -- (0.2946410161513772,2.0862315948301524) -- cycle;
\draw[line width=1pt] (-1.48,5.08) -- (-4.38,3.38) -- (-4.66,5.84) -- cycle;
\draw[line width=1pt] (-4.66,5.84) -- (-2.28,8.16) -- (-1.48,5.08) -- cycle;
\draw[line width=1pt] (2,5.12) -- (4.92,3.44) -- (5.68,5.56) -- cycle;
\draw[line width=1pt] (2,5.12) -- (5.68,5.56) -- (3.32,8.24) -- cycle;
\draw [line width=1pt] (2,5.12)-- (-1.48,5.08);
\draw [line width=1pt] (-1.48,5.08)-- (0.2946410161513772,2.0862315948301524);
\draw [line width=1pt] (0.2946410161513772,2.0862315948301524)-- (2,5.12);
\draw [line width=1pt] (-1.48,5.08)-- (-2.22,2.02);
\draw [line width=1pt] (-2.22,2.02)-- (0.2946410161513772,2.0862315948301524);
\draw [line width=1pt] (0.2946410161513772,2.0862315948301524)-- (-1.48,5.08);
\draw [line width=1pt] (2,5.12)-- (2.96,2.04);
\draw [line width=1pt] (2.96,2.04)-- (0.2946410161513772,2.0862315948301524);
\draw [line width=1pt] (0.2946410161513772,2.0862315948301524)-- (2,5.12);
\draw [line width=1pt] (-1.48,5.08)-- (-4.38,3.38);
\draw [line width=1pt] (-4.38,3.38)-- (-4.66,5.84);
\draw [line width=1pt] (-4.66,5.84)-- (-1.48,5.08);
\draw [line width=1pt] (-4.66,5.84)-- (-2.28,8.16);
\draw [line width=1pt] (-2.28,8.16)-- (-1.48,5.08);
\draw [line width=1pt] (-1.48,5.08)-- (-4.66,5.84);
\draw [line width=1pt] (2,5.12)-- (4.92,3.44);
\draw [line width=1pt] (4.92,3.44)-- (5.68,5.56);
\draw [line width=1pt] (5.68,5.56)-- (2,5.12);
\draw [line width=1pt] (2,5.12)-- (5.68,5.56);
\draw [line width=1pt] (5.68,5.56)-- (3.32,8.24);
\draw [line width=1pt] (3.32,8.24)-- (2,5.12);
\draw [line width=1pt,dash pattern=on 2pt off 2pt] (-2.22,2.02)-- (-3,2.08)-- (-3.64,2.32)-- (-4.12,2.78)-- (-4.38,3.38);
\draw [line width=1pt,dash pattern=on 2pt off 2pt] (2.96,2.04)-- (3.68,2.18)-- (4.2,2.46)-- (4.6,2.92)-- (4.92,3.44);
\draw [line width=1pt,dash pattern=on 2pt off 2pt] (-1.48,5.08)-- (-4.12,2.78);
\draw [line width=1pt,dash pattern=on 2pt off 2pt] (-3.64,2.32)-- (-1.48,5.08);
\draw [line width=1pt,dash pattern=on 2pt off 2pt] (-1.48,5.08)-- (-3,2.08);
\draw [line width=1pt,dash pattern=on 2pt off 2pt] (2,5.12)-- (3.68,2.18);
\draw [line width=1pt,dash pattern=on 2pt off 2pt] (4.2,2.46)-- (2,5.12);
\draw [line width=1pt,dash pattern=on 2pt off 2pt] (2,5.12)-- (4.6,2.92);
\draw [fill=black] (2,5.12) circle (2.5pt);
\draw[color=black] (1.7,5.4) node {$C$};
\draw [fill=black] (-1.48,5.08) circle (2.5pt);
\draw[color=black] (-1.3,5.4) node {$B$};
\draw [fill=black] (0.2946410161513772,2.0862315948301524) circle (2.5pt);
\draw[color=black] (0.5,1.8) node {$A_0$};
\draw [fill=black] (-2.22,2.02) circle (2.5pt);
\draw[color=black] (-1.8985160575858246,1.7) node {$A_1$};
\draw [fill=black] (2.96,2.04) circle (2.5pt);
\draw[color=black] (3.0969213732004444,1.7) node {$A_{-1}$};
\draw [fill=black] (-4.38,3.38) circle (2.5pt);
\draw[color=black] (-5,3.7) node {$A_{q-5}$};
\draw [fill=black] (-4.66,5.84) circle (2.5pt);
\draw[color=black] (-5.3,5.8) node {$A_{q-4}$};
\draw [fill=black] (-2.28,8.16) circle (2.5pt);
\draw[color=black] (-1.7,8.187176079734211) node {$A_{q-3}$};
\draw [fill=black] (4.92,3.44) circle (2.5pt);
\draw[color=black] (5.6,3.7) node {$A_{-q+5}$};
\draw [fill=black] (5.68,5.56) circle (2.5pt);
\draw[color=black] (6.3,5.6) node {$A_{-q+4}$};
\draw [fill=black] (3.32,8.24) circle (2.5pt);
\draw[color=black] (4.2,8.144296788482828) node {$A_{-q+3}$};
\draw [fill=black] (-3,2.08) circle (2.5pt);
\draw [fill=black] (-3.64,2.32) circle (2.5pt);
\draw [fill=black] (-4.12,2.78) circle (2.5pt);
\draw [fill=black] (3.68,2.18) circle (2.5pt);
\draw [fill=black] (4.2,2.46) circle (2.5pt);
\draw [fill=black] (4.6,2.92) circle (2.5pt);
\end{tikzpicture}
\caption{First step in lower bound construction}\label{fig:construction}
\end{figure}
Then the points $\{B,C\}\cup\{A_i: -q+3\leq i\leq q-3\}$ form a penny graph as long as $A_{q-3}C,BA_{-q+3},A_{q-3}A_{-q+3}\geq d$.
It is sufficient to ensure that $A_{q-3}A_{-q+3}\geq d$, since then the triangle inequality will ensure that $A_{q-3}C+BA_{-q+3}\geq A_{q-3}A_{-q+3}+BC\geq 2d$, hence $A_{q-3}C=BA_{-q+3}\geq d$.
The inequality $A_{q-3}A_{-q+3}\geq d$ is equivalent to 
\begin{equation}\label{1}
\beta_q \leq 2\pi-(q-2)\alpha,
\end{equation}
where $\beta_q=\angle A_{q-3}A_{-q+3}C=\angle A_{-q+3}A_{q-3}B$.
We fix $q$ to be the maximum value such that \eqref{1} holds.
Thus if we also construct equilateral triangles $\triangle A_{q-3}A_{q-2}B$ and $\triangle A_{-q+3}A_{-q+2}C$,
then we have
\begin{equation}\label{1'}
\beta_{q+1} > 2\pi-(q-1)\alpha.
\end{equation}
Since $A_{q-3}A_{-q+3}\geq d > A_{q-2}A_{-q+2}$, it follows that there exists a square $ABCD$ where $A$ is inside the angle $\angle A_{q-3}BA_{q-2}$ and $D$ is inside the angle $\angle A_{-q+3}CA_{-q+2}$.
Thus \eqref{1} and \eqref{1'} can be improved to
\begin{equation}\label{1''}
\beta_q \leq\alpha_4\leq 2\pi-(q-2)\alpha,
\end{equation}
and
\begin{equation}\label{1'''}
\beta_{q+1} > \alpha_4 > 2\pi-(q-1)\alpha,
\end{equation} where $\alpha_4$ is the angle of a square of side length $d$.
Thus $q=2+\lfloor(2\pi-\alpha_4)/\alpha\rfloor$.

One further constraint that we need is $(q-2)\alpha > \pi$,
which implies that $A_{q-3}$ and $A_0$ are on opposite sides of the line through $B$ and $C$,
as implicitly assumed in Figure~\ref{fig:construction}.
This follows from \eqref{1'''} and the inequalities $\alpha < \pi/3$ and $\alpha_4 < \pi/2$.

So far, we have $2q-3$ vertices and $4q-9$ edges.
For the next step, reflect everything in the line $\ell_1$ through $A_{q-3}$ and $A_{-q+3}$.
Then we get $2q-5$ reflected vertices and $4q-9$ reflected edges.
In order for the reflected points to be at distance at least $d$ from the original points, it is sufficient for the orginal points $B$ and $A_0,\dots,A_{q-4}$ to be at distance at least $d/2$ from the line $\ell_1$.
To show that $B$ is at distance at least $d/2$ from $\ell_1$, it is sufficient to show that
\begin{equation}\label{2}
2\beta_q\geq\alpha,
\end{equation}
since this will give that $BB'\geq d$ in the triangle $\triangle BA_{q-3}B'$, where $B'$ is the reflection of $B$ in $\ell_1$.
Similarly, considering the triangle $A_{q-4}A_{q-3}A_{q-4}'$ (where $A_{q-4}'$ is the reflection of $A_{q-4}$ in $\ell_1$), to show that $A_{q-4}$ is at distance at least $d/2$ from $\ell_1$, it is sufficient to show that
\begin{equation}\label{3}
\pi-(\alpha+\beta_q)\geq\alpha/2.
\end{equation}
Given \eqref{2} and \eqref{3}, $B$, $C$ and $A_{q-4}$ are at distance at least $d/2$ from $\ell_1$, and the hypercycle (equidistant curve) at distance $d/2$ from the line $\ell_1$ on the side of $B$ and $C$, intersects the circle with centre $B$ and radius $d$ in points on the arc from $A_{q-4}$ to $B$ that contains $A_{q-3}$.
It then follows that all the points $A_1,\dots,A_{q-5}$ are at distance at least $d/2$ from $\ell_1$ as they are outside this arc.

Before finishing the construction, we show \eqref{2} and \eqref{3}.
Let $E$ be such that $\triangle ABE$ is equilateral and $E$ and $C$ are on opposite sides of the line $AB$ (Figure~\ref{fig:lower_bound}).
\begin{figure}
\centering
\begin{tikzpicture}[line cap=round,line join=round,>=triangle 45,x=1cm,y=1cm,scale=2]
\clip(-3.2,-1.3) rectangle (1.3,1.8);
\coordinate (A) at (-1,1);
\coordinate (B) at (-1,-1);
\coordinate (C) at (1,-1);
\coordinate (D) at (1,1);
\coordinate (E) at (-2.732050807568877,0);
\coordinate (Aq3) at (-1.845236523481399,0.8126155740732999);
\coordinate (Aq4) at (-2.9923893961834915,-0.8256885145046834);
\coordinate (P) at (0,0.8126155740732999);
\coordinate (Q) at (0,0);
\coordinate (bottom) at (0,-1.3);
\coordinate (top) at (0,1.8);
\draw[line width=1pt] (A) -- (B) -- (C) -- (D) -- cycle;
\draw[line width=1pt] (B) -- (A) -- (E) -- cycle;
\draw[line width=1pt] (B) -- (Aq3) -- (Aq4) -- cycle;
\draw[line width=1pt] (Aq3) -- (P);
\draw[line width=1pt] (E) -- (Q);
\draw [line width=.8pt,dotted] (bottom) -- (top);
\node[above left] at (A) {$A$};
\node[below left] at (B) {$B$};
\node[below right] at (C) {$C$};
\node[above right] at (D) {$D$};
\node[below right] at (top) {$m$};
\node[right] at (P) {$P$};
\node[right] at (Q) {$Q$};
\node[above] at (Aq3) {$A_{q-3}$};
\node[left] at (E) {$E$};
\node[below] at (Aq4) {$A_{q-4}$};
\draw[red!80!black, thick] pic["$\beta_q$", draw=red!80!black, angle radius=14mm, angle eccentricity=.7] {angle=B--Aq3--P};
\draw[blue, thick] pic["$\frac{\alpha}{2}$", draw=blue, angle radius=14mm, angle eccentricity=.8] {angle=B--E--Q};
\draw[green!60!black, thick] pic["$\cdot$", draw=green!60!black, angle radius=4mm, angle eccentricity=.6] {angle=E--Q--bottom};
\draw[green!60!black, thick] pic["$\cdot$", draw=green!60!black, angle radius=4mm, angle eccentricity=.6] {angle=Aq3--P--bottom};
\begin{scriptsize}
\draw [fill=black] (-1,1) circle (1pt);
\draw [fill=black] (-1,-1) circle (1pt);
\draw [fill=black] (1,-1) circle (1pt);
\draw [fill=black] (1,1) circle (1pt);
\draw [fill=black] (-2.732050807568877,0) circle (1pt);
\draw [fill=black] (0,0) circle (1pt);
\draw [fill=black] (-1.845236523481399,0.8126155740732999) circle (1pt);
\draw [fill=black] (-1,-1) circle (1pt);
\draw [fill=black] (-2.9923893961834915,-0.8256885145046834) circle (1pt);
\draw [fill=black] (0,0.8126155740732999) circle (1pt);
\end{scriptsize}
\end{tikzpicture}
\caption{}\label{fig:lower_bound}
\end{figure}
Since $A$ is inside the angle $\angle A_{q-3}BA_{q-2}$, it follows that $A_{q-3}$ is inside the angle $\angle EBA$.
Let $P$ be the foot of the perpendicular from $A_{q-3}$ to the perpendicular bisector $m$ of $BC$, and $Q$ the foot of the perpendicular from $E$ to $m$.
Then $\beta_q=\angle BA_{q-3}P \geq \angle BEQ = \alpha/2$, which proves \eqref{2}.
Since $\beta_q\leq\alpha_4$ and $\alpha\leq\pi/3$, we have $\beta_q+3\alpha/2\leq\pi$, which is equivalent to~\eqref{3}.

Let $A_0'$ be the reflection of $A_0$ and $A_1'$ be the reflection of $A_1$ in $\ell_1$.
Build up triangles around $A_0'$ and $A_1'$ just as was done originally around $B$ and $C$.
This creates an additional $q-5$ triangles around $A_0'$, with an additional $q-5$ vertices and $2q-10$ additional edges, and an additional $q-4$ triangles around $A_1'$ with an additional $q-4$ vertices and $2q-8$ additional edges.
This figure has a point symmetry around the midpoint of $B'A_0'$.
Thus if we again reflect in a line $\ell_2$, then $\ell_1$ and $\ell_2$ are point reflections of each other, hence are ultraparallel.

After $i$ iterations, we have added $(2q-5)i$ vertices and $4q-9$ edges from reflection in the lines $\ell_1$ to $\ell_i$, and another $(2q-9)(i-1)$ vertices and $(4q-18)(i-1)$ edges from constructing triangles,
to obtain \[n_{i}=2q-3 + (2q-5)i+(2q-9)(i-1)\] vertices and \[e_{i}=4q-9+(4q-9)i+(4q-18)(i-1)\] edges.
Eliminating $i=(n-6)/(4q-14)$, we obtain that as long as $n-6$ is divisible by $4q-14$, we have constructed a hyperbolic penny graph with \[e=\left(2+\frac{1}{4q-14}\right)n-3+\frac{14}{4q-14}\] edges.
We can fill in the inbetween values of $n$ by repeatedly adding vertices of degree $2$.
This gives us at least $(2+\frac{1}{4q-14})n-4$ edges in the worst case where $n-6$ is one less than a multiple of $4q-14$.
Since any two consecutive reflecting lines $\ell_i$ and $\ell_{i+1}$ are ultraparallel, all reflecting lines are pairwise ultraparallel, so there is never any overlapping in any iteration.
\end{proof}

\bibliography{bibliography}
\bibliographystyle{abbrv}

\end{document}